\documentclass[12pt]{amsart}
\usepackage{amscd}
\usepackage{graphicx}
\input xy
\xyoption{all}
\def\frk{\mathfrak}               

\def\Phi{{\frk N}}
\def\opn#1#2{\def#1{\operatorname{#2}}} 
\opn\chara{char} \opn\length{\ell} \opn\pd{pd} \opn\rk{rk}
\opn\projdim{proj\,dim} \opn\injdim{inj\,dim} \opn\rank{rank}
\opn\depth{depth} \opn\grade{grade} \opn\height{height}
\opn\embdim{emb\,dim} \opn\codim{codim}

\opn\Tr{Tr} \opn\bigrank{big\,rank}
\opn\superheight{superheight}\opn\lcm{lcm}
\opn\trdeg{tr\,deg}
\opn\reg{reg} \opn\lreg{lreg} \opn\ini{in} \opn\lpd{lpd}
\opn\size{size}\opn{\mult}{mult}
\opn\div{div} \opn\Div{Div} \opn\cl{cl} \opn\Cl{Cl}
\opn\Spec{Spec} \opn\Supp{Supp} \opn\supp{supp} \opn\Sing{Sing}
\opn\Ass{Ass} \opn\Min{Min}
\opn\Ann{Ann} \opn\Rad{Rad} \opn\Soc{Soc}
\opn\Syz{Syz} \opn\Im{Im} \opn\Ker{Ker} \opn\Coker{Coker}
\opn\Am{Am} \opn\Hom{Hom} \opn\Tor{Tor} \opn\Ext{Ext}
\opn\End{End} \opn\Aut{Aut} \opn\id{id} \opn\ini{in}

\opn\nat{nat}
\opn\pff{pf}
\opn\Pf{Pf} \opn\GL{GL} \opn\SL{SL} \opn\mod{mod} \opn\ord{ord}
\opn\Gin{Gin}
\opn\Hilb{Hilb}\opn\adeg{adeg}\opn\std{std}\opn\ip{infpt}
\opn\Pol{Pol}
\opn\sat{sat}
\opn\Var{Var}
\opn\Gen{Gen}

\opn\aff{aff} \opn\con{conv} \opn\relint{relint} \opn\st{st}
\opn\lk{lk} \opn\cn{cn} \opn\core{core} \opn\vol{vol}
\opn\link{link} \opn\star{star}
\opn\gr{gr}

\def\pot#1#2{#1[\kern-0.28ex[#2]\kern-0.28ex]}

\opn\dirlim{\underrightarrow{\lim}}
\opn\inivlim{\underleftarrow{\lim}}
\def\Implies{\ifmmode\Longrightarrow \else
        \unskip${}\Longrightarrow{}$\ignorespaces\fi}
\def\implies{\ifmmode\Rightarrow \else
        \unskip${}\Rightarrow{}$\ignorespaces\fi}
\def\iff{\ifmmode\Longleftrightarrow \else
        \unskip${}\Longleftrightarrow{}$\ignorespaces\fi}

\let\:=\colon
\newtheorem{Theorem}{Theorem}[section]
\newtheorem{Lemma}[Theorem]{Lemma}
\newtheorem{Corollary}[Theorem]{Corollary}

\newtheorem{Remark}[Theorem]{Remark}

\newtheorem{Example}[Theorem]{Example}

\newtheorem{Definition}[Theorem]{Definition}

\let\epsilon\varepsilon
\let\phi=\varphi
\let\kappa=\varkappa
\def\qed{\ifhmode\textqed\fi
      \ifmmode\ifinner\quad\qedsymbol\else\dispqed\fi\fi}
\def\textqed{\unskip\nobreak\penalty50
       \hskip2em\hbox{}\nobreak\hfil\qedsymbol
       \parfillskip=0pt \finalhyphendemerits=0}
\def\dispqed{\rlap{\qquad\qedsymbol}}

\opn\dis{dis}
\def\pnt{{\raise0.5mm\hbox{\large\bf.}}}

\opn\Lex{Lex}

\begin{document}

\UseRawInputEncoding

\title{Characterization and Newton Complementary Dual of Quasi $f$-Ideals  $^{*}$  }

\author{ F. U. Rehman $^{1}$, H. Hasan $^{1}$, H. Mahmood $^{1}$, M. A. Binyamin $^{2}$
}
\thanks{\noindent $^{*}$ The third and fourth authors are supported by the Higher Education Commission of Pakistan for this research (Grant no. 7515).\\
\noindent $^{1}$Government College University Lahore, Pakistan.
$^{2}$ Government College University  Faisalabad, Pakistan.\\
{\em E-mails }: fazalqau@gmail.com, hasanmahmood@gcu.edu.pk, halahasan8783@gmail.com,
 ahsanbanyamin@gmail.com  }
\maketitle
\begin{abstract}
 The notion of quasi $f$-ideals was first presented in $[14]$ which generalize the idea of $f$-ideals.  In this paper, we give the complete characterization of quasi $f$-ideals of degree greater or equal to $2$. Additionally, we show that the property of being quasi $f$-ideals remain the same after taking the Newton complementary dual of a squarefree monomial ideal $I$ provided that the minimal generating set of $I$ is perfect.

 \vskip 0.4 true cm
\noindent
\end{abstract}

\section{Introduction}
Throughout this paper,  $F$ is a field and $R=F[x_{1},x_{2,}...,x_{n}]$ is a polynomial ring with $n$ indeterminate. Any squarefree monomial ideal $I\subset R$ can be associated to two different simplicial complexes over the finite set of vertices, denoted by $\delta _{\mathcal{F}}(I)$ and $\delta _{\mathcal{N}}(I)$, called the facet complex of $I$ and the non-face complex (or Stanley-Reisner complex) of $I$ respectively. The $f$-vectors of these two simplicial complexes $\delta _{\mathcal{F}}(I)$ and $\delta _{\mathcal{N}}(I)$ have the accompanying prospects :\\1) $ f(\delta _{\mathcal{F}}(I))= f(\delta _{\mathcal{N}}(I))$ or\\ 2) $f(\delta _{\mathcal{F}}(I))\neq  f(\delta _{\mathcal{N}}(I)) $ but $dim(\delta _{\mathcal{F}}(I))= dim(\delta _{\mathcal{N}}(I))$ or \\ 3)  $f(\delta _{\mathcal{F}}(I))\neq  f(\delta _{\mathcal{N}}(I)) $ but $dim(\delta _{\mathcal{F}}(I))\neq  dim(\delta _{\mathcal{N}}(I))$ \\ A squarefree monomial ideal $I\subset R$ with property that mention in (1) is called an $f$-ideal of the polynomial ring $R$. This notion has been studied for various properties of $f$-ideals relevant to combinatorial commutative algebra in the papers $[3]$, $[4]$, $[11]$, $[12]$, and $[13]$. If we look at (1) and (2) collectively, there is one thing common, i.e, dimensions of both $\delta _{\mathcal{F}}(I)$ and $\delta _{\mathcal{N}}(I)$ are same. This means that the $f$-vectors of $\delta _{\mathcal{F}}(I)$ and $\delta _{\mathcal{N}}(I)$ can be added or subtracted usually. In this paper we will think just those squarefree monomial ideal $I$ in the polynomial ring $R=F[x_{1},x_{2,}...,x_{n}]$ with properties that $dim(\delta _{\mathcal{F}}(I))=s=dim(\delta _{\mathcal{N}}(I))$ and $f(\delta _{\mathcal{N}}(I))- f(\delta _{\mathcal{F}}(I))=\left(a_1,a_2,\ldots,a_s\right)$, we call it a quasi $f$-ideal of type $(a_1,a_2,\ldots,a_s)$. It is noted that if $a_{i}=0$ for all $i=1,2,...,r$, then clearly we have the property (1), and equivalently saying that every $f$-ideal is a quasi $f$-ideal of type $\mathbf{0}$-vector. It is natural to ask: Is it possible to characterize all the squarefree monomial quasi $f$-ideals? Mahmood. H. et al, classified all the pure squarefree monomial quasi $f$-ideals of degree $2$ in two different approaches in $[14]$. In this paper, we characterize all quasi $f$-ideals of degree $d\geq 2$. Moreover, we extend this class of quasi $f$-ideal regarding its Newton complementary dual. In 2013, during the study of Cremona maps; Costa and Simis $[5]$ introduced the notion of the Newton complementary dual in general context. After that, Doria and Simis $[6]$ were examined different properties of Newton complementary dual. Ansaldi, Lin, and Shin $[2]$ investigated the Newton complementary duals of monomial ideals.\\

\indent In this paper, we use the following outlines. In section 2, we give essential concepts that help in the expected results. Section 3 focuses on the primary study of quasi $f$-ideals; Theorem 3.2 characterizes all pure squarefree quasi $f$-ideals of degree $d\geq 2$, and Theorem 3.2 gives the formulae to compute the Hilbert function and Hilbert series of the polynomial ring $R$ modulo squarefree monomial quasi $f$-ideal $I$. The fourth section of this paper is devoted to the Newton complementary duals of quasi $f$-ideals. We prove that an ideal $I$ is a quasi $f$-ideal if and only if $\widehat{I}$ is a quasi $f$-ideal provided that $G(I)$ is perfect, Theorem 4.5.

\section{Basic Set Up}
Let us review some fundamental ideas to get familiar with simplicial complex and squarefree monomial ideals.
Let $V=\{v_1,v_2,\ldots,v_n\}$ be a vertex set. A subset $\Delta$ of $P(V)$ is said to be a simplicial complex on $V$ if and only if each one point subset of $V$ lies in $\Delta$ and, if $E$ is any subset of  $F\in \Delta$, then $E\in \Delta$. Note that, each element of $\Delta$ is known as face and the maximal faces under $\subseteq$ are known as facets. For any face $F$ of $\Delta$, the dimension of $F$ is given by $dim(F)=|F|-1$. Note that $\emptyset\in \Delta$ and $dim(\emptyset)=-1$. A simplicial complex $\Delta$ is said to be pure if all of its facets have the same dimension. The dimension of $\Delta$ is the maximum of the dimension of all facets of $\Delta$. If $\Delta$ is a $d$-dimensional simplicial complex on $V$, then the $f$-vector of $\Delta$ is the $d+2$ tuple $$f(\Delta)=(f_{-1}(\Delta), f_{0}(\Delta),f_{1}(\Delta),...,f_{d}(\Delta))$$ where $f_{-1}(\Delta)=1$ and $f_{i}(\Delta)=\left\vert \left\{ F\in \Delta :\dim (F)=i\right\}\right\vert $, $0\leq i\leq d$,. Note that in $d$-dimensional simplicial complex $\Delta$, we may assume if needed $f_{i}(\Delta)=0$  for $i>d$. This means that the $f$-vector of simplicial complex $\Delta$ can be written as $f(\Delta)=(f_{-1}(\Delta), f_{0}(\Delta),f_{1}(\Delta),...,f_{d}(\Delta), 0, 0, ...,0)$.\\

\indent Suppose that $I=\langle u_{1},u_{2},...,u_{r}\rangle$ is a squarefree monomial ideal of a polynomial ring $R=F[x_1,x_2, \ldots,x_n]$ with $supp(I)=\{x_1,x_2, \ldots,x_n\}$. We use $G(I)$ to denote the unique set of minimal generators of $I$. We can associate to $I$ two simplicial complexes. The facet complex is a simplicial complex denoted by $\delta _{\mathcal{F}}(I)$ and is defined as $\delta _{\mathcal{F}}(I)=\langle F_{1},F_{2},..., F_{r}\rangle$ where $F_{i}=\{v_{j}:x_{j}$ divide
$u_{i}\}$ is the facet on the vertices $v_1,v_2,\ldots,v_n$, and $i=1,2,3,...,r$.
The non-face complex (or Stanley-Reisner complex) is a simplicial complex on $V=\{v_1,v_2,\ldots v_n\}$ such that a subset $\{v_{i_{1}},v_{i_{2}},...,v_{i_{k}}\}$ of $V$ is a face of this non-face complex if and only if the corresponding monomial
$x_{i_{1}}x_{i_{2}}\cdots x_{i_{k}}$ does not belong to $I$. We denote it by $\delta_{\mathcal{N}}(I)$. In this paper, we are interested in the following family of squarefree monomial ideals:
\begin{Definition}
\emph{A squarefree monomial ideal $I$ in the polynomial ring $R$ with the field $F$ is said to be a quasi $f$-ideal
of type $(a_1,a_2,\ldots,a_r)\in \mathbb{Z}^r $ if and only if $f_{i}(\delta _{\mathcal{N}}(I))- f_{i}(\delta _{\mathcal{F}}(I))=a_i$ for all $i\in \{1,2,3,...,r\}.$}
\end{Definition}
\begin{Remark}\emph{It is noted that, if $a_{i}=0$ for all $i\in \{1,2,3,...,r\}$, then obviously $I$ is an $f$-ideal. This means that every $f$-ideal is a quasi $f$-ideal whose type is a zero vector, and obviously any quasi $f$-ideal with type some non-zero vector can not be an $f$-ideal. It is impotent to mention that the class of quasi $f$-ideals is much more bigger than the class of $f$-ideals; moreover, unlike $f$-ideals, examples of quasi $f$-ideals can be found in $R=F[x_{1},x_{2,}...,x_{n}]$, for any $n$.}
\end{Remark}

We would like to recall the [12, Definition 2.1] of perfect sets of $R$. We use $sm(R)$ and $sm(R)_d$ to denote the set of all squarefree monomials and the set of all squarefree monomials of degree $d$ in $R$ respectively;
\begin{Definition}\emph{Let $R=F[x_{1},x_{2,}...,x_{n}]$, and let $A\subseteq sm(R)$. We define \\
(i) $\sqcup(A)=\{gx_{i} \ | \ g\in A, {x_{i}} \text{  does not divide  } g, 1\leq
i\leq n\}$ and \\ (ii) $\sqcap(A)=\{h \ | \ h=g/x_{i} \text{  for some  } g\in A  \text{  and some  } x_{i} \text{
with  } x_{i}|g\}$\\
$A$ is called lower perfect if $\sqcap(A)= sm(R)_{d-1}$ holds. Dually, $A$ is called upper perfect if $\sqcup(A)= sm(R)_{d+1}$ holds. If $A$ is both lower perfect and upper perfect, then $A$ is called a perfect set.}
\end{Definition}

\section{Characterization of Quasi $f$-Ideals}
The purpose of this section is to characterize quasi $f$-ideals of degree $d\geq2$ in the polynomial ring $R=F[x_{1},x_{2,}...,x_{n}]$.

%
The following lemma is key to understanding how to count the number of faces for particular dimensions of non-face complex of a pure squarefree monomial ideal for any degree.
\begin{Lemma}
\emph{Let $I$ be a pure squarefree monomial ideal of degree $d+1$ in
$R=F[x_{1},x_{2},...,x_{n}]$. Then $f_{i}(\delta _{\mathcal{N}}(I))$
$=$ $n \choose {i+1}$, where $0\leq i<{d}.$}

\end{Lemma}

\begin{proof}
If $n_i$ is the number of $i$-dimensional non-faces of facet
complex of $I$ where $0\leq i<{d}$, then it means that $f_{i}(\delta _{\mathcal{F}}(I))$ $=$ $n \choose {i+1}$ $-$
$n_i$ and $\delta_{\mathcal{N}}(I)$ contains at least $n_i$ number of $i$-dimensional faces. The [3, Lemma 3.6] tells us that the face of dimension less than $d$ of $\delta_{\mathcal{F}}(I)$ is the face of $\delta_{\mathcal{N}}(I)$. Therefore, we have ${f_{i}(\delta _{\mathcal{N}}(I))}={n \choose {i+1}}-{n_i}+{n_i}={n \choose {i+1}}$

\end{proof}

\begin{Theorem}(Characterization)\\
\emph{Let $I$ be an equigenerated squarefree
monomial ideal of $R$ of degree $d$ and let $G(I)=\{u_1,u_2,\ldots,u_r\}$ be minimal set of generator of $I$. Then
$I$ is quasi $f$-ideal of type $(a_{0}, a_{1}, a_{2},..., a_{d-1})\in \mathbb{Z}^d$ if and only if the following conditions
hold true:
\begin{enumerate}
    \item $ht(I)=n-d$;
    \item ${n\choose d}\equiv  \begin{cases}
                           0\  (mod\  2) & \text{if $a_{d-1}$ is even} \\
                           1\  (mod\  2) & \text{if $a_{d-1}$ is odd}
                         \end{cases}
,$ and $r=|G(I)|=\frac{1}{2}({n\choose d}-a_{d-1})$. ;
    \item $a_{i}$ is the number of $i$-dimensional non-faces of $\delta _{\mathcal{F}}(I)$, where $i=0, 1, 2,..., d-2$.
\end{enumerate}}
\end{Theorem}

\begin{proof}
If $I$ is a quasi $f$-ideal of degree $d$ with type $(a_{0}, a_{1}, a_{2},...,
a_{d-1})$, then obviously, dimensions of both facet complex and non-face complex of $I$ are same. By [1, Lemma 3.4] the height of ideal $I$ must be equal to $n-d$. As $I$ is a squarefree monomial ideal of degree $d$, using the [1, Lemma 3.2], we have that
\begin{equation}
    \label{eq-1}  f_{d-1}(\delta _{\mathcal{F}}(I))={n \choose d}-f_{d-1}(\delta_{\mathcal{N}}(I))\end{equation}
It is noted that if $I$ is a quasi $f$-ideal of type $(a_{0}, a_{1}, a_{2},..., a_{d-1})$, then
\begin{equation}
    \label{eq-1}f_{d-1}(\delta _{\mathcal{N}}(I))= f_{d-1}(\delta
_{\mathcal{F}}(I))+a_{d-1}\end{equation} From equation (1) and equation (2), we have that $2f_{d-1}(\delta _{\mathcal{F}}(I))={n \choose d}-a_{d-1}$. It means that the parity of ${n\choose d}$ is same as the parity of $a_{d-1}$, and $|G(I)|=f_{d-1}(\delta _{\mathcal{F}}(I))=\frac{1}{2}({n\choose d}-a_{d-1})$. Now we want to show that $a_{i}$ is the number of $i$-dimensional non-faces of $\delta _{\mathcal{F}}(I)$, where $i=0, 1, 2,..., d-2$.
Note that for any squarefree monomial ideal $I$ of degree $d$ the Lemma 3.1, tells us that the number of $i$-dimensional faces of $\delta_{\mathcal{N}}(I)$ will be equal to $n \choose {i+1}$ for all $0\leq i <d-1$. As $I$ is a quasi $f$-ideal of degree $d$ with type $(a_{0}, a_{1}, a_{2},..., a_{d-1})$. Therefore, $f_{i}(\delta_{\mathcal{N}}(I))- f_{i}(\delta _{\mathcal{F}}(I))=a_{i}$ which implies ${n \choose {i+1}}-f_{i}(\delta _{\mathcal{F}}(I))=a_{i}$ for all
$0\leq i <d-1$.\\

\indent Conversely, suppose that conditions $(1)$, $(2)$ and $(3)$ hold true. For any squarefree monomial ideal $I$, [15, Proposition
5.3.10] gives $dim (\delta _{\mathcal{N}}(I))=n-ht(I)-1$ and since height of $I$ is $n-d$, so we have $dim(\delta_{\mathcal{N}}(I))=dim(\delta _{\mathcal{F}}(I))$. Now, we will show that $f(\delta_{\mathcal{N}}(I))-f(\delta_{\mathcal{N}}(I))=(a_{0}, a_{1}, a_{2},..., a_{d-1}).$ Since $I$ is a squarefree monomial ideal of degree $d$, $f_{d-1}(\delta_{\mathcal{N}}(I))={n \choose d}-f_{d-1}(\delta _{\mathcal{F}}(I))$. This implies \begin{equation}
    \label{eq-1}f_{d-1}(\delta_{\mathcal{N}}(I))-f_{d-1}(\delta_{\mathcal{F}}(I))={n \choose d}-2f_{d-1}(\delta _{\mathcal{F}}(I))\end{equation} As $f_{d-1}(\delta _{\mathcal{F}}(I))=\frac{1}{2}({n\choose d}-a_{d-1})$, this implies ${n\choose d}-{2f_{d-1}(\delta _{\mathcal{F}}(I))}={a_{d-1}}$. Therefore, equation (3) becomes  $f_{d-1}(\delta _{\mathcal{N}}(I))-f_{d-1}(\delta _{\mathcal{F}}(I))= a_{d-1}$.
Since $a_{i}$ is the number of $i$-dimensional non-faces of $\delta _{\mathcal{F}}(I)$ where $i=0, 1, 2,..., d-2$, this means that ${n\choose {i+1}}-f_{i}(\delta _{\mathcal{F}}(I))=a_{i}$. From Lemma $3.1$, the number of $i$-dimensional faces of  $\delta _{\mathcal{N}}(I))$ are ${n\choose {i+1}}$ therefore, we have $f_{i}(\delta _{\mathcal{N}}(I))-f_{i}(\delta _{\mathcal{F}}(I))=a_{i}$ for all $i\in \{0, 1, 2,..., d-2\}$. Thus $I$ is quasi $f$-ideal of type $(a_{0}, a_{1}, a_{2},..., a_{d-1})\in \mathbb{Z}^d$.
\end{proof}

It is worth noting that [3, Theorem 3.9] and [14, Theorem 4.1] can be deduce by setting $a_{i}=0$ for all $i\in \{0,1,...,d-1\}$ in Theorem 3.2 such as the following corollary:

\begin{Corollary}
\emph{Let $I=(u_1,u_2,\ldots,u_r)$ be an equigenerated squarefree
monomial ideal of $R$ of degree $d.$ Then $I$ is quasi $f$-ideal of type $\mathbf{0}=(0,0,...,0)\in \mathbb{Z}^d$ if and only if the following conditions
hold true:
\begin{enumerate}
    \item $ht(I)=n-d$;
    \item ${n\choose d}\equiv 0\  (mod\  2),$ and $r=|G(I)|=\frac{1}{2}{n\choose d}$;
    \item $f_{d-2}(\delta _{\mathcal{F}}(I))={n\choose {d-1}}$ \
\end{enumerate}}

\end{Corollary}

\begin{proof}
It is clear from above Theorem. Note that if $a_{d-2}=0$, then $r=|G(I)|=\frac{1}{2}{n\choose d}$. From Lemma 3.1 and [1, Lemma 3.2], together implies $f_{d-2}(\delta _{\mathcal{F}}(I))={n\choose {d-1}}.$
\end{proof}

\begin{Corollary}
\emph{Let $I$ be an equigenerated squarefree
monomial ideal of $R$ of degree $2$ and let $G(I)=\{u_1,u_2,\ldots,u_r\}$ be minimal set of generator of $I$. Then
$I$ is quasi $f$-ideal of type $(a_{0}, a_{1})\in \mathbb{Z}^2$ if and only if the following conditions
hold true:
\begin{enumerate}
    \item $ht(I)=n-2$;
    \item ${n\choose 2}\equiv  \begin{cases}
                           0\  (mod\  2) & \text{if $a_{1}$ is even} \\
                           1\  (mod\  2) & \text{if $a_{1}$ is odd}
                         \end{cases}
,$ and $r=|G(I)|=\frac{1}{2}({n\choose 2}-a_{1})$. ;
    \item $a_{0}$ is the number of $i$-dimensional non-faces of $\delta _{\mathcal{F}}(I)$.
\end{enumerate}}

\end{Corollary}

\begin{Example}
\emph{Let
$I=(x_{1}x_{2}x_{4},x_{1}x_{2}x_{5},x_{3}x_{4}x_{5},x_{1}x_{4}x_{5},x_{2}x_{3}x_{5})$
be a pure squarefree monomial ideal of degree $3$ in the polynomial
ring $F[x_{1},x_{2},x_{3},x_{4},x_{5}].$ Then the primary
decomposition of $I$ is $(x_{1},x_{3})\bigcap (x_{1},x_{5})\bigcap
(x_{2},x_{4})\bigcap (x_{2},x_{5})\bigcap (x_{4},x_{5}).$ The facet complex
and the non-face complex of $I$ are $$\delta _{\mathcal{F}}(I)=\langle
\{1,2,4\},\{1,2,5\},\{3,4,5\},\{1,4,5\},\{2,3,5\}\rangle$$ and
$$\delta _{\mathcal{N}}(I)=\langle
\{1,2,3\},\{1,3,4\},\{1,3,5\},\{2,3,4\},\{2,4,5\}\rangle.$$ As
$\{1,3\}$ is non-face of $\delta _{\mathcal{F}}(I)$,
$ht(I)=2$, ${5\choose 3}\equiv 0\ (mod\ 2)$ and
$|G(I)|=5=\frac{1}{2}({5\choose 3}-0).$ Thus $I$ is a quasi
$f$-ideal with type $(0,0,1,0).$}
\end{Example}

\begin{Example}
\emph{Let
$I=(x_{1}x_{2}x_{4},x_{1}x_{2}x_{5},x_{3}x_{4}x_{5},x_{1}x_{4}x_{5})$
be a pure squarefree monomial ideal of degree $3$ in the polynomial
ring $F[x_{1},x_{2},x_{3},x_{4},x_{5}].$ Then the primary
decomposition of $I$ is $(x_{1},x_{3})\bigcap (x_{1},x_{4})\bigcap
(x_{1},x_{5})\bigcap (x_{2},x_{4})\bigcap (x_{2},x_{5})\bigcap
(x_{4},x_{5}).$ The facet and non-face complexes of $I$ are $$\delta
_{\mathcal{F}}(I)=\langle
\{1,2,4\},\{1,2,5\},\{3,4,5\},\{1,4,5\}\rangle$$ and
$$\delta _{\mathcal{N}}(I)=\langle
\{1,2,3\},\{1,3,4\},\{1,3,5\},\{2,3,4\},\{2,3,5\},\{2,4,5\}\rangle.$$
Then $f(\delta _{\mathcal{F}}(I))=(1,5,8,4)$ and $f(\delta
_{\mathcal{N}}(I))=(1,5,10,6).$ Thus $I$ is a quasi $f$-ideal with
type $(0,0,2,2).$}
\end{Example}

In the following theorem, we give the formulae to compute the Hilbert function and Hilbert series of polynomial ring $R$ modulo a squarefree monomial quasi $f$-ideal I.

\begin{Theorem}
\emph{Let $I$ be a quasi $f$-ideal in a polynomial ring $R$
of type $(a_{-1}, a_{0}, a_{1},\\...,a_{d}).$ Then \\ (1) The Hilbert series of $R/I$ is given by
$$F(R\slash I,k)=\sum_{i=-1}^{d}{\frac{f_{i}(\delta _{\mathcal{F}}(I))+a_{i}}{{(1-k)}^{i+1}}{k^{i+1}}}$$\\ (2) The Hilbert function of $R/I$ is given by
$$H(R\slash I,k)=\sum_{i=0}^{d}{{{k-1}\choose i}(f_{i}(\delta _{\mathcal{F}}(I))+a_{i})}$$ where $k\geq 1$ and $H(R\slash I,0)=1$.}
\end{Theorem}

\begin{proof}
As $I$ is a quasi $f$-ideal in a polynomial $R$ of type $(a_{-1}, a_{0}, a_{1},..., a_{d})$. Therefore, $f_{i}(\delta_{\mathcal{N}}(I))= f_{i}(\delta_{\mathcal{F}}(I))+a_{i}$, this means that the $f$-vector of $f(\delta _{\mathcal{N}}(I))$ is $(f_{-1}(\delta _{\mathcal{F}}(I))+a_{-1},f_{0}(\delta _{\mathcal{F}}(I))+a_{0},...,f_{d}(\delta _{\mathcal{F}}(I))+a_{d}).$ Since $I$ is  non face ideal of the simplicial complex $\delta _{\mathcal{N}}(I))$. Now using [15, Theorem 6.7.2] and [15, Theorem 6.7.3], we have that $$F(R\slash I,k)=\sum_{i=-1}^{d}{\frac{f_{i}(\delta _{\mathcal{F}}(I))+a_{i}}{{(1-k)}^{i+1}}{k^{i+1}}}$$ and $$H(R\slash I,k)=\sum_{i=0}^{d}{{{k-1}\choose i}(f_{i}(\delta _{\mathcal{F}}(I))+a_{i})}$$ as desired.
\end{proof}

\begin{Example}
\emph{We return to the ideal $I$ in Example 3.6. This ideal $I$ is a quasi $f$-ideal with type $(0,0,2,2)$ (i.e. $a_{-1}=0, a_{0}=0,a_{1}=2$ and $a_{2}=2$) and the facet complex of $I$ is $$\delta_{\mathcal{F}}(I)=\langle \{1,2,4\},\{1,2,5\},\{3,4,5\},\{1,4,5\}\rangle$$ clearly, $f(\delta _{\mathcal{F}}(I))=(1,5,8,4)$ (i.e. $f_{-1}=1, f_{0}=5,f_{1}=8$ and $f_{2}=4$. Thus the Hilbert function and Hilbert series of $R[x_{1},x_{2},x_{3},x_{4},x_{5}]/I$ are $$H(R/I,k)={{k-1}\choose 0}(5+0)+{{k-1}\choose 1}(8+2)+{{k-1}\choose 2}(4+2)$$ and $$F(R/I,k)={\frac{1+0}{{(1-k)}^{-1+1}}{k^{-1+1}}}+{\frac{5+0}{{(1-k)}^{0+1}}{k^{0+1}}}+{\frac{8+2}{{(1-k)}^{1+1}}{k^{1+1}}}+{\frac{4+2}{{(1-k)}^{2+1}}{k^{2+1}}}$$ respectively, where $R=F[x_{1},x_{2},x_{3},x_{4},x_{5}].$}
\end{Example}

\section{Newton complementary dual of Quasi $f$-Ideal}

 In [4], Budd and Van Tuyl examined that the property of being $f$-ideals remain the same after taking the Newton complementary dual of a squarefree monomial ideal $I$. As quasi $f$-ideal is the generalization of  $f$-ideal. Therefore, it is natural to ask the following question: \\ When is Newton's complementary dual of $I$ a quasi $f$-ideal?\\ In this section, we have tended to this inquiry. Firstly, we recall the [4, Definition 3.1] of Newton complementary dual of a squarefree monomial ideal.
\begin{Definition}
\emph{Let $I$ be a monomial ideal of $R$ with $G(I)=\{u_1,u_2,\ldots,u_r\}$ is a minimal set of generators of $I$. The Newton complementary dual of $I$ is simply denoted by $\widehat{I}$ and is defined as $$\widehat{I}=\langle \frac {\prod\limits_{i=1}^{n}{x_{i}}}{u}|u\in G(I)\rangle$$}
\end{Definition}

\begin{Example}
\emph{The Newton complementary dual of an ideal $I$ in Example 3.5 is $\widehat I=(x_{3}x_{5},x_{3}x_{4},x_{1}x_{2},x_{2}x_{3},x_{1}x_{4}).$}
\end{Example}
It is easy to understand that $\widehat I=(x_{3}x_{5},x_{3}x_{4},x_{1}x_{2},x_{2}x_{3},x_{1}x_{4})$ is not a quasi $f$-ideal. Because the minimal generating set $G(\widehat I)$ of $\widehat I$ is not upper perfect since $x_{2}x_{4}x_{5}\notin {\sqcup G(\widehat I)} $. But keeping in mind that $I$ is a quasi $f$-ideal. It means that the Newton complementary dual of a quasi $f$-ideal need not be quasi $f$-ideal. The following remark helps us to overcome the problem of duality of a squarefree monomial quasi $f$-ideal.

\begin{Remark}
\emph{For a squarefree monomial ideal $I$ of degree $d$ in a polynomial ring $R=F[x_{1},x_{2},...,x_{n}]$, $G(I)$ is lower (upper) perfect set if and only if $G(\widehat I)$ is upper (lower) perfect. Indeed, If $G(I)$ is lower perfect and suppose that there is a squarefree monomial $m\notin {\widehat I} $ of degree $n-d+1$. Then we must have a squarefree monomial $u\notin {I}$ of degree $d-1$ , which is contradiction to the fact that $G(I)$ is lower perfect set. This means that $G(I)$ is perfect set if and only if $G(\widehat I)$ is perfect set. The perfectness of $G(I)$ guaranties that the dimensions of both, the facet complex of $\widehat I$ and the non-face complex of $\widehat I$, coincide.}
\end{Remark}
Now, we recall [4, Corollary 3.6].

\begin{Corollary}
\emph{Let $I$ be a squarefree monomial ideal of $R$.\\ (1) If $f(\delta
_{\mathcal{F}}(I))=(f_{-1},f_{0},...,f_{d})$, then $$f(\delta_{\mathcal{N}}(\widehat I))=({n\choose 0}-f_{n-1},...,{n\choose i}-f_{n-i-1}),...,{n\choose {n-1}}-f_{0},{n\choose n}-f_{-1})$$ \\ (2) If $f(\delta
_{\mathcal{N}}(I))=(f_{-1},f_{0},...,f_{d})$, then $$f(\delta_{\mathcal{F}}(\widehat I))=({n\choose 0}-f_{n-1},...,{n\choose i}-f_{n-i-1}),...,{n\choose {n-1}}-f_{0},{n\choose n}-f_{-1})$$
In both cases, $f_{i} = 0$ if $i > d$.}
\end{Corollary}

\begin{Theorem}
 \emph{Let $I$ be a squarefree monomial ideal of degree $d$ in a polynomial ring $R=F[x_{1},x_{2},...,x_{n}]$ and let the minimal generating set $G(I)$ of $I$ be perfect set. Then $I$ is a quasi $f$-ideal with type $(a_{-1}, a_{0}, a_{1},...,a_{d-1})\in \mathbb{Z}^{d+1}$, where $a_{i}=0$ for all $i\geq d$ if and only if $\widehat{I}$ is a quasi $f$-ideal of type $(a_{n-1}, a_{n-2}, a_{n-3}, ..., a_{-1}).$}
\end{Theorem}

\begin{proof}
If $I$ is a quasi $f$-ideal with type $(a_{-1}, a_{0}, a_{1},...,a_{d-1})\in \mathbb{Z}^{d+1}$, then $f_{i}(\delta _{\mathcal{F}}(I))= f_{i}(\delta_{\mathcal{N}}(I))-a_{i}$ for all $i\in \{-1.0,1,...,d-1\}$. This means that the $f$-vector of facet complex of $I$ will be $(f_{-1}(\delta_{\mathcal{N}}(I))-a_{-1}, (f_{0}(\delta_{\mathcal{N}}(I))-a_{0}, ..., (f_{d-1}(\delta_{\mathcal{N}}(I))-a_{d-1})$. Now, using [4, Corollary 3.6], we have
\begin{equation}
    \label{eq-1}f_{j}(\delta _{\mathcal{N}}(\widehat{I}))={n\choose {j+1}}- f_{n-j-2}(\delta_{\mathcal{N}}(I))+a_{n-j-2}\end{equation} where $j\in \{-1.0,1,...,n-1\}$. It is noted that if $(f_{-1}, f_{0}, ..., f_{d-1})$ is $f$-vector of $\delta_{\mathcal{N}}(I)$, then from [4, Corollary 3.6],  we get \begin{equation}
    \label{eq-1}f_{j}(\delta _{\mathcal{F}}(\widehat{I}))={n\choose {j+1}}- f_{n-j-2}(\delta_{\mathcal{N}}(I))\end{equation} where $j\in \{-1.0,1,...,n-1\}$. From equation (3) and equation (4), we have
    \begin{equation}
    \label{eq-1}f_{j}(\delta _{\mathcal{N}}(\widehat{I}))=f_{j}(\delta_{\mathcal{F}}(\widehat{I}))+a_{n-j-2}\end{equation} This implies $f_{j}(\delta _{\mathcal{N}}(\widehat{I}))-f_{j}(\delta_{\mathcal{F}}(\widehat{I}))=a_{n-j-2}$ for all $j\in \{-1.0,1,...,n-1\}$. This shows that $\widehat{I}$ is a quasi $f$-ideal of type $(a_{n-1}, a_{n-2}, a_{n-3},..., a_{-1}).$ Similarly for reveres implication, we simply replace $I$ with $\widehat{I}$.

\end{proof}

\begin{Corollary}
\emph{Let $I$ be a squarefree monomial ideal of a polynomial ring $R=F[x_{1},x_{2},...,x_{n}]$. Then $I$ is a quasi $f$-ideal of type $\textbf{0}$-vector if and only if $\widehat{I}$ is a quasi $f$-ideal of type $\textbf{0}$-vector.}

\end{Corollary}
\begin{proof}
It is obvious. Since,  a quasi $f$-ideal $I$ of type $\textbf{0}$-vector is an $f$-ideal. It implies that the minimal generating set $G(I)$ of $I$ is perfect. Hence, we have needed to show follows from the above theorem.

\end{proof}

\begin{Example}
\emph{Let us consider the ideal $I$ in the polynomial ring $F[x_{1},x_{2},x_{3},x_{4},x_{5},x_{6}]$ given by $I=(x_{1}x_{2},x_{1}x_{3},x_{2}x_{3},x_{4}x_{5},x_{4}x_{6},x_{5}x_{6},x_{2}x_{4},x_{3}x_{5})$. It is easy to see that $I$ is a quasi $f$-ideal of type $(0,0,-1)$ $($ or it can be written as $(0,0,-1,0,0,0,0),$ here $a_{-1}=0$, $a_{0}=0$, $a_{1}=-1$, $a_{2}=0$, $a_{3}=0$, $a_{4}=0$ and $a_{5}=0)$. \\ Clearly the Newton complementary dual $\widehat{I}$ of $I$ is generated by the monomials $ \{x_{3}x_{4}x_{5}x_{6},x_{2}x_{4}x_{5}x_{6},x_{1}x_{4}x_{5}x_{6},x_{1}x_{2}x_{3}x_{6},x_{1}x_{2}x_{3}x_{5}, x_{1}x_{2}x_{3}x_{4},x_{1}x_{3}x_{5}x_{6},x_{1}x_{2}x_{4}x_{6}\}$. The primary decomposition of $\widehat{I}$ is $(x_{2},x_{5})\bigcap (x_{3},x_{4})\bigcap (x_{2},x_{6})\bigcap (x_{3},x_{6})\bigcap (x_{1},x_{4})\bigcap \\(x_{1},x_{5})\bigcap (x_{1},x_{6})\bigcap (x_{1},x_{2},x_{3})\bigcap (x_{4},x_{5},x_{6})$.
It is easy to calculate that $f(\delta _{\mathcal{F}}(\widehat{I}))=(1,6,15,20,8)$ and $f(\delta _{\mathcal{N}}(\widehat{I}))=(1,6,15,20,7)$. Thus $\widehat{I}$ is a quasi $f$-ideal with type $(0,0,0,0,-1)$ instead of $(0,0,0,0,-1,0,0)\leftrightarrow (a_{5},a_{4},a_{3},a_{2},a_{1},a_{0},a_{-1})$}
\end{Example}

\end{document}